\documentclass[a4paper,12pt]{amsart}
\usepackage{amssymb}
\usepackage{amsmath}
\usepackage[all]{xy}

\usepackage{hyperref}

\theoremstyle{plain}
{
  \newtheorem{thm}{Theorem}[section]
  \newtheorem{defn}[thm]{Definition}
  
  \newtheorem{lem}[thm]{Lemma}

}

\newcommand{\R}{{\rm R}}

\newcommand{\ad}{ad}

\DeclareMathOperator{\GL}{GL}

\newcommand{\Spec}{\text{\rm Spec}\,}
\DeclareMathOperator{\Hom}{Hom}
\DeclareMathOperator{\Aut}{Aut}
\DeclareMathOperator{\Cent}{Cent}

\DeclareMathOperator{\Sym}{Sym}
\DeclareMathOperator{\Ga}{{\mathbf G}_a}
\newcommand{\id}{\text{\rm id}}

\DeclareMathOperator{\ZZ}{{\mathbb Z}}
\DeclareMathOperator{\QQ}{{\mathbb Q}}

\DeclareMathOperator{\Lie}{Lie}

\usepackage{hyperref}

\newcommand{\st}{\scriptstyle}
\DeclareMathOperator{\Gm}{{\mathbb G}_m}

\begin{document}

\title[Grothendieck--Serre conjecture over semi-local Dedekind domains]
{On the Grothendieck--Serre conjecture concerning principal $G$-bundles over semi-local Dedekind domains}

\author{I.~Panin}
\thanks{Theorem~\ref{th:main} is proved due to the support of the
Russian Science Foundation (grant no. 14-11-00456).}
\address{Steklov Institute of Mathematics at St.-Petersburg, Fontanka 27, St.-Petersburg 191023, Russia}
\email{paniniv@gmail.com}

\author{A.~Stavrova}
\thanks{The second author is a postdoctoral fellow of the program
6.50.22.2014 ``Structure theory, representation theory and geometry of algebraic groups''
at St. Petersburg State University.}
\address{Department of Mathematics and Mechanics, St. Petersburg State University,
St. Petersburg, Russia}
\email{anastasia.stavrova@gmail.com}

\maketitle

\begin{abstract}
Let $R$ be a semi-local Dedekind domain
and let $K$ be the field of fractions of R. Let $G$ be a
reductive semisimple simply connected $R$-group scheme
such that every semisimple normal $R$-subgroup scheme of $G$
contains a split $R$-torus $\mathbb G_{m,R}$.
We prove that
the kernel of the map
$$ H^1_{\text{\'et}}(R,G)\to H^1_{\text{\'et}}(K,G) $$
\noindent
induced by the inclusion of $R$ into $K$, is trivial.
This result partially extends the Nisnevich theorem~\cite[Thm.4.2]{Ni}.

\end{abstract}

\section{Introduction}
A well-known conjecture due to J.-P.~Serre and A.~Grothendieck
\cite[Remarque, p.31]{Se},
\cite[Remarque 3, p.26-27]{Gr1},
and
\cite[Remarque 1.11.a]{Gr2}
asserts that given a regular local ring $R$ and its field of
fractions $K$ and given a reductive group scheme $G$ over $R$ the
map
$$ H^1_{\text{\'et}}(R,G)\to H^1_{\text{\'et}}(K,G), $$
\noindent
induced by the inclusion of $R$ into $K$, has trivial kernel.

The Grothendieck--Serre conjecture holds for semi-local regular
rings containing a field. That is proved in \cite{FP} and in \cite{Pa1}.
The first of these two papers is heavily based on results of
\cite{PSV} and \cite{Pa2}. For the detailed history of the topic see, for instance,
\cite{FP}. Assuming that $R$ is not equicharacteristic, the conjecture has been established only in the case
where $G$ is an $R$-torus~\cite{C-T-S} and in the case where $G$ is a reductive group
scheme over a discrete valuation ring $R$~\cite[Theorem 4.2]{Ni}.
In the present paper, we extend the latter result to the case of an isotropic semisimple simply connected 
reductive group scheme over a semi-local Dedekind domain $R$;
see Theorem~\ref{th:main}.

\section{Preliminaries}

\subsection{Parabolic subgroups and elementary subgroups}

Let $A$ be a commutative ring. Let $G$ be an isotropic reductive group scheme over $A$, and
let $P$ be a parabolic subgroup of $G$ in the sense of~\cite{SGA3}.
Since the base $\Spec A$ is affine, the group $P$ has a Levi subgroup $L_P$~\cite[Exp.~XXVI Cor.~2.3]{SGA3}.
There is a unique parabolic subgroup $P^-$ in $G$ which is opposite to $P$ with respect to $L_P$,
that is $P^-\cap P=L_P$, cf.~\cite[Exp. XXVI Th. 4.3.2]{SGA3}.  We denote by $U_P$ and $U_{P^-}$ the unipotent
radicals of $P$ and $P^-$ respectively.

\begin{defn}
\label{defn:E_P}
The \emph{elementary subgroup $E_P(A)$ corresponding to $P$} is the subgroup of $G(A)$
generated as an abstract group by $U_P(A)$ and $U_{P^-}(A)$.
\end{defn}

Note that if $L'_P$ is another Levi subgroup of $P$,
then $L'_P$ and $L_P$ are conjugate by an element $u\in U_P(A)$~\cite[Exp. XXVI Cor. 1.8]{SGA3}, hence
$E_P(A)$ does not depend on the choice of a Levi subgroup or of an opposite subgroup
$P^-$, respectively. We suppress the particular choice of $L_P$ or $P^-$ in this context.

\begin{defn}
A parabolic subgroup $P$ in $G$ is called
\emph{strictly proper}, if it intersects properly every normal semisimple subgroup of $G$.
\end{defn}

We  will use the following result that is a combination
of~\cite{PS} and~\cite[Exp. XXVI, \S 5]{SGA3}.

\begin{lem}\label{lem:EE}
Let $G$ be a reductive group scheme over a commutative ring $A$, and let $R$ be a commutative $A$-algebra.
Assume that $A$ is a semilocal ring. Then the subgroup $E_P(R)$ of $G(R)$ is the same for any
minimal parabolic $A$-subgroup $P$ of $G$. If, moreover, $G$ contains a strictly proper parabolic $A$-subgroup,
the subgroup $E_P(R)$ is the same for any strictly proper parabolic $A$-subgroup $P$.
\end{lem}
\begin{proof}
See~\cite[Theorem 2.1]{St-poly}.
\end{proof}

\subsection{Torus actions on reductive groups}

Let $R$ be a commutative ring with 1, and let $S=(\Gm_{,R})^N=\Spec(R[x_1^{\pm 1},\ldots,x_N^{\pm 1}])$
be a split $N$-dimensional torus over $R$. Recall that the character group
$X^*(S)=\Hom_R(S,\Gm_{,R})$ of $S$ is canonically isomorphic to $\ZZ^N$.
If $S$ acts $R$-linearly on an $R$-module $V$, this module has a natural $\ZZ^N$-grading
$$
V=\bigoplus_{\lambda\in X^*(S)}V_\lambda,
$$
where
$$
V_\lambda=\{v\in V\ |\ s\cdot v=\lambda(s)v\ \mbox{for any}\ s\in S(R)\}.
$$
Conversely, any $\ZZ^N$-graded $R$-module $V$ can be provided with an $S$-action by the same rule.

Let $G$ be a reductive group scheme over $R$ in the sense of~\cite{SGA3}. Assume that $S$ acts on $G$
by $R$-group automorphisms. The associated Lie algebra functor $\Lie(G)$ then acquires
a $\ZZ^N$-grading compatible with the Lie algebra structure,
$$
\Lie(G)=\bigoplus_{\lambda\in X^*(S)}\Lie(G)_\lambda.
$$

We will use the following version of~\cite[Exp. XXVI Prop. 6.1]{SGA3}.

\begin{lem}\label{lem:T-P}
Let $L=\Cent_G(S)$ be the subscheme of $G$ fixed by $S$. Let
$\Psi\subseteq X^*(S)$ be an $R$-subsheaf of sets closed under addition of characters.

(i) If $0\in\Psi$, then there exists
a unique smooth connected closed subgroup $U_\Psi$ of $G$ containing $L$ and satisfying
\begin{equation}\label{eq:LieUPsi}
\Lie(U_\Psi)=\bigoplus_{\lambda\in\Psi}\Lie(G)_\lambda.
\end{equation}
Moreover, if $\Psi=\{0\}$, then $U_\Psi=L$; if $\Psi=-\Psi$, then $U_\Psi$ is reductive; if $\Psi\cup(-\Psi)=X^*(S)$,
then $U_\Psi$ and $U_{-\Psi}$ are two opposite parabolic subgroups of $G$ with the common Levi subgroup
$U_{\Psi\cap(-\Psi)}$.

(ii) If $0\not\in\Psi$, then there exists a unique smooth connected unipotent closed subgroup $U_\Psi$ of $G$
normalized by $L$ and satisfying~\eqref{eq:LieUPsi}.
\end{lem}
\begin{proof}
The statement immediately follows by faithfully flat descent from the standard facts about the subgroups of
split reductive groups proved in~\cite[Exp. XXII]{SGA3}; see the proof of~\cite[Exp. XXVI Prop. 6.1]{SGA3}.
\end{proof}

\begin{defn}
The sheaf of sets
$$
\Phi=\Phi(S,G)=\{\lambda\in X^*(S)\setminus\{0\}\ |\ \Lie(G)_\lambda\neq 0\}
$$
is called the \emph{system of relative roots of $G$ with respect to $S$}.
\end{defn}

Choosing a total ordering on the $\QQ$-space $\QQ\otimes_{\ZZ} X^*(S)\cong\QQ^n$, one defines the subsets
of positive and negative relative roots $\Phi^+$ and $\Phi^-$, so that $\Phi$ is a disjoint
union of $\Phi^+$, $\Phi^-$, and $\{0\}$. By Lemma~\ref{lem:T-P} the closed subgroups
$$
U_{\Phi^+\cup\{0\}}=P,\qquad U_{\Phi^-\cup\{0\}}=P^-
$$
are two opposite parabolic subgroups of $G$ with the common Levi subgroup $\Cent_G(S)$.
Thus, if a reductive group $G$ over $R$ admits a non-trivial action of a split torus,
then it has a proper parabolic subgroup. The converse is true Zariski-locally, see~Lemma~\ref{lem:relroots} below.

\subsection{Relative roots and subschemes}

In order to prove our main result, we need to use the notions of relative roots and relative root
subschemes. These notions were initially introduced and studied in~\cite{PS}, and further developed
in~\cite{St-serr}.

Let $R$ be a commutative ring. Let $G$ be a reductive group scheme over $R$. Let $P$ be a parabolic subgroup
scheme of $G$ over $R$, and let $L$ be a Levi subgroup of $P$.
By~\cite[Exp. XXII, Prop. 2.8]{SGA3} the root system $\Phi$ of $G_{\overline{k(s)}}$, $s\in\Spec R$,
is constant locally in the Zariski topology on $\Spec R$. The type of the root system of
$L_{\overline{k(s)}}$ is determined by a Dynkin subdiagram
of the Dynkin diagram of $\Phi$, which is also constant Zariski-locally on $\Spec R$
by~\cite[Exp. XXVI, Lemme 1.14 and Prop. 1.15]{SGA3}. In particular, if $\Spec R$ is connected,
all these data are constant on $\Spec R$.

\begin{lem}\label{lem:relroots}\cite[Lemma 3.6]{St-serr}
Let $G$ be a reductive group over a connected commutative ring $R$, $P$ be a parabolic subgroup of $G$, $L$ be a Levi
subgroup of $P$, and $\bar L$ be the image of $L$ under the natural
homomorphism $G\to G^{\ad}\subseteq \Aut(G)$. Let $D$ be the Dynkin diagram of the root system $\Phi$ of
$G_{\overline{k(s)}}$ for any $s\in\Spec A$. We identify $D$ with a set of simple roots of $\Phi$ such that
$P_{\overline{k(s)}}$ is a standard positive parabolic subgroup with respect to $D$. Let $J\subseteq D$
be the set of simple roots such that $D\setminus J\subseteq D$ is the subdiagram corresponing to $L_{\overline{k(s)}}$.
Then there are a unique maximal split subtorus
$S\subseteq\Cent(\bar L)$ and a subgroup $\Gamma\le \Aut(D)$ such that $J$ is invariant under $\Gamma$,
and
for any $s\in\Spec R$ and any split maximal torus $T\subseteq\bar L_{\overline{k(s)}}$
the kernel of the natural surjection
\begin{equation}\label{eq:T-S}
X^*(T)\cong\ZZ\Phi\xrightarrow{\ \pi\ } X^*(S_{\overline{k(s)}})\cong \ZZ\Phi(S,G)
\end{equation}
is generated by all roots $\alpha\in D\setminus J$,
and by all differences $\alpha-\sigma(\alpha)$, $\alpha\in J$, $\sigma\in\Gamma$.
\end{lem}

In~\cite{PS}, we introduced a system of relative roots $\Phi_P$ with respect to a parabolic
subgroup $P$ of a reductive group $G$ over a commutative ring $R$. This system $\Phi_P$ was defined
independently over each member $\Spec A=\Spec A_i$
of a suitable finite disjoint Zariski covering
$$
\Spec R=\coprod\limits_{i=1}^m\Spec A_i,
$$
such that over each $A=A_i$, $1\le i\le m$, the root system $\Phi$ and the Dynkin diagram $D$ of $G$ is constant.
Namely, we considered the formal projection
$$
\pi_{J,\Gamma}\colon\ZZ \Phi
\longrightarrow \ZZ\Phi/\left<D\setminus J;\ \alpha-\sigma(\alpha),\ \alpha\in J,\ \sigma\in\Gamma\right>,
$$
and set $\Phi_P=\Phi_{J,\Gamma}=\pi_{J,\Gamma}(\Phi)\setminus\{0\}$. The last claim of Lemma~\ref{lem:relroots}
allows to identify $\Phi_{J,\Gamma}$ and $\Phi(S,G)$ whenever $\Spec R$ is connected.

\begin{defn}
In the setting of Lemma~\ref{lem:relroots} we call $\Phi(S,G)$ a \emph{system
of relative roots with respect to the parabolic subgroup $P$ over $R$} and denote it by $\Phi_P$.
\end{defn}

If $A$ is a field or a local ring, and $P$ is a minimal parabolic subgroup of $G$,
then $\Phi_P$ is nothing but the relative root system of $G$ with respect to a maximal split subtorus
in the sense of~\cite{BoTi} or, respectively,~\cite[Exp. XXVI \S 7]{SGA3}.

We have also defined in~\cite{PS} irreducible components of systems of relative roots, the subsets of positive and negative
relative roots, simple relative roots, and the height of a root. These definitions are immediate analogs
of the ones for usual abstract root systems, so we do not reproduce them here.

Let $R$ be a commutative ring with 1.
For any finitely generated projective $R$-module $V$, we denote by $W(V)$ the natural affine scheme
over $R$ associated with $V$, see~\cite[Exp. I, \S 4.6]{SGA3}.
Any morphism of $R$-schemes $W(V_1)\to W(V_2)$
is determined by an element $f\in\Sym^*(V_1^\vee)\otimes_R V_2$, where $\Sym^*$ denotes the symmetric algebra,
and $V_1^\vee$ denotes the dual module of $V_1$. If $f\in\Sym^d(V_1^\vee)\otimes_R V_2$,
we say that the corresponding morphism is
homogeneous of degree $d$.
By abuse of notation, we also write $f:V_1\to V_2$ and call it {\it a degree $d$
homogeneous polynomial map from $V_1$ to $V_2$}. In this context, one has
$$
f(\lambda v)=\lambda^d f(v)
$$
for any $v\in V_1$ and $\lambda\in R$.

\begin{lem}\cite[Lemma 3.9]{St-serr}\label{lem:relschemes}.
In the setting of Lemma~\ref{lem:relroots}, for any $\alpha\in\Phi_P=\Phi(S,G)$ there exists a closed
$S$-equivariant  embedding of $R$-schemes
$$
X_\alpha\colon W\bigl(\Lie(G)_\alpha\bigr)\to G,
$$
satisfying the following condition.

\begin{itemize}
\item[\bf{($*$)}] Let $R'/R$ be any ring extension such that $G_{R'}$ is split with
respect to a maximal split $R'$-torus $T\subseteq L_{R'}$. Let $e_\delta$,
$\delta\in\Phi$, be a Chevalley basis of $\Lie(G_{R'})$, adapted to $T$ and $P$, and $x_\delta\colon\Ga\to G_{R'}$, $\delta\in\Phi$, be the associated
system of 1-parameter root subgroups
{\rm(}e.g. $x_\delta=\exp_\delta$ of~\cite[Exp. XXII, Th. 1.1]{SGA3}{\rm)}.
Let
$$
\pi:\Phi=\Phi(T,G_{R'})\to\Phi_P\cup\{0\}
$$
be the natural projection.
Then for any
$u=\hspace{-8pt}\sum\limits_{\delta\in\pi^{-1}(\alpha)}\hspace{-8pt}a_\delta e_\delta\in\Lie(G_{R'})_\alpha$
one has
\begin{equation}\label{eq:Xalpha-prod}
X_\alpha(u)=
\Bigl(\prod\limits_{\delta\in\pi^{-1}(\alpha)}\hspace{-8pt}x_{\delta}(a_\delta)\Bigr)\cdot
\prod\limits_{i\ge 2}\Bigl(\prod\limits_{
\st
\theta\in \pi^{-1}(i\alpha)
}
\hspace{-8pt}x_\theta(p^i_{\theta}(u))\Bigr),
\end{equation}
where every $p^i_{\theta}:\Lie(G_{R'})_\alpha\to R'$ is a homogeneous polynomial map of degree $i$,
and the products over $\delta$ and $\theta$ are taken in any fixed order.
\end{itemize}
\end{lem}

\begin{defn}
Closed embeddings $X_\alpha$, $\alpha\in\Phi_P$, satisfying the statement of Lemma~\ref{lem:relschemes},
are called \emph{relative root subschemes of $G$ with respect to the parabolic subgroup $P$}.
\end{defn}

Relative root subschemes of $G$ with respect to $P$, actually,
depend on the choice of a Levi subgroup $L$ in $P$, but their essential properties stay the same,
so we usually omit $L$ from the notation.

We will use the following properties of relative root subschemes.

\begin{lem}\label{lem:rootels}\cite[Theorem 2, Lemma 6, Lemma 9]{PS}
Let $X_\alpha$, $\alpha\in\Phi_P$, be as in Lemma~\ref{lem:relschemes}.
Set $V_\alpha=\Lie(G)_\alpha$ for short. Then

(i) There exist degree $i$ homogeneous polynomial maps $q^i_\alpha:V_\alpha\oplus V_\alpha\to V_{i\alpha}$, i>1,
such that for any $R$-algebra $R'$ and for any
$v,w\in V_\alpha\otimes_R R'$ one has
\begin{equation}\label{eq:sum}
X_\alpha(v)X_\alpha(w)=X_\alpha(v+w)\prod_{i>1}X_{i\alpha}\left(q^i_\alpha(v,w)\right).
\end{equation}

(ii) For any $g\in L(R)$, there exist degree $i$ homogeneous polynomial maps
$\varphi^i_{g,\alpha}\colon V_\alpha\to V_{i\alpha}$, $i\ge 1$, such that for any $R$-algebra $R'$ and for any
$v\in V_\alpha\otimes_R R'$ one has
$$
gX_\alpha(v)g^{-1}=\prod_{i\ge 1}X_{i\alpha}\left(\varphi^i_{g,\alpha}(v)\right).
$$

(iii) \emph{(generalized Chevalley commutator formula)} For any $\alpha,\beta\in\Phi_P$
such that $m\alpha\neq -k\beta$ for all $m,k\ge 1$,
there exist polynomial maps
$$
N_{\alpha\beta ij}\colon V_\alpha\times V_\beta\to V_{i\alpha+j\beta},\ i,j>0,
$$
homogeneous of degree $i$ in the first variable and of degree $j$ in the second
variable, such that for any $R$-algebra $R'$ and for any
for any $u\in V_\alpha\otimes_R R'$, $v\in V_\beta\otimes_R R'$ one has
\begin{equation}\label{eq:Chev}
[X_\alpha(u),X_\beta(v)]=\prod_{i,j>0}X_{i\alpha+j\beta}\bigl(N_{\alpha\beta ij}(u,v)\bigr)
\end{equation}

(iv) For any subset $\Psi\subseteq X^*(S)\setminus\{0\}$ that is closed under addition,
the morphism
$$
X_\Psi\colon W\Bigl(\,\bigoplus_{\alpha\in\Psi}V_\alpha\Bigr)\to U_\Psi,\qquad
(v_\alpha)_\alpha\mapsto\prod_\alpha X_\alpha(v_\alpha),
$$
where the product is taken in any fixed order,
is an isomorphism of schemes.
\end{lem}

\begin{lem}\label{lem:EPgen}
In the notation of Lemma~\ref{lem:relroots}, let $\Phi^\pm$ be the set of positive and negative roots
such that $D\subseteq\Phi^+$. Set $\Phi_P^\pm=\pi(\Phi^\pm)\setminus\{0\}$, $P^+=P$,
and let $P^-$ be the opposite parabolic subgroup
to $P$ such that $P\cap P^-=L$. Then for any $R$-algebra $R'$, one has
$$
U_{P^\pm}(R')=\left<X_\alpha(R'\otimes_R V_\alpha),\ \alpha\in\Phi_P^\pm\right>.
$$
Consequently,
$$
E_P(R')=\left<X_\alpha(R'\otimes_R V_\alpha),\ \alpha\in\Phi_P\right>.
$$
\end{lem}
\begin{proof}
By the choice of $D$ the parabolic subgroup $P_{\overline{k(s)}}$ is the standard positive parabolic subgroup
of $G_{\overline{k(s)}}$ corresponding to a closed set of roots $\Psi\supseteq\Phi^+$. By the choice of
$J\subseteq D$, one has
$$
\Psi=\Phi^+\cup\bigl(\ZZ(D\setminus J)\cap\Phi^-\bigr).
$$
Then, clearly, $\pi(\Psi)=\Phi_P^+\cup\{0\}$. Similarly, $P^-$ corresponds to the set $(-\Psi)$ and
$\pi(-\Psi)=\Phi_P^-\cup\{0\}$. Then the unipotent radicals $U_{P^\pm}$ correspond
to the closed unipotent subsets
$$
\pi\bigl(\Phi^\pm\setminus \ZZ(D\setminus J)\bigr)=\Phi_P^\pm\subseteq \Phi_P.
$$
Then Lemma~\ref{lem:rootels} (iv) finishes the proof.

\end{proof}

\section{Main Theorem}

All commutative rings are assumed to be unital.
For any commutative ring $R$ and $n\ge 3$, we denote by $E_n(R)$ the usual elementary subgroup of $\GL_n(R)$.






\begin{lem}\label{lem:EEn}
Let $R$ be a commutative ring, Let $G$ be a reductive group scheme over $R$, and let $i:G\to\GL_{n,R}$
be a closed embedding of $G$ as a closed $R$-subgroup, where $n\ge 3$.
Assume that $G$ contains a non-central 1-dimensional subtorus $H\cong\Gm_{,R}$, and let $P=P^+$ and
$P^-$ be the corresponding two opposite parabolic subgroups constructed as in Lemma~\ref{lem:T-P}.
 Then one has $E_P(R)\le E_n(R)$.
\end{lem}
\begin{proof}
Let $Q=Q^+$ and $Q^-$ be the two parabolic $R$-subgroups
of $\GL_{n,R}$ corresponding to $H\le G\le\GL_{n,R}$, and let $M=\Cent_{\GL_{n,R}}(T)$ be their
common Levi subgroup. We show that $U_P(R)\le U_Q(R)$. Clearly, this implies the claim of the lemma.
By~\cite[Proposition 2.8.3(3)]{CGP} this is true if $R$ is a field. In general, take $g\in U_P(R)$.
It is enough to show that $g\in U_Q(R_m)$ for any maximal localization $R_m$ of $R$. Let
$$
\rho:R_m\to R_m/mR_m=l
$$
be the residue homomorphism. By the above $\rho^*(g)\in U_Q(l)$.
Recall that $\Omega_Q=U_{Q^+}MU_{Q^-}\cong U_{Q^+}\times M\times U_{Q^-}$ is an open subscheme of
$\GL_{n,R}$~\cite[Exp. XXVI, Remarque 4.3.6]{SGA3}. Hence
\begin{equation}\label{eq:gOmegaQ}
g\in U_{Q^+}(R_m)M(R_m)U_{Q^-}(R_m).
\end{equation}

Let $L=P\cap P^-=\Cent_G(H)$ be the Levi subgroup of $P$ and $P^-$.
Let $\bar H\subseteq G^{\ad}$ be the image of $H$ under the natural homomorphism $G\to G^{\ad}$. Clearly,
$\bar H\cong \Gm_{,R_m}$ is a split subtorus of the center of the image $\bar L$ of $L$ in $G^{\ad}$. Let
$S\le\Cent(\bar L_{R_m})$ be the split torus constructed in Lemma~\ref{lem:relroots} (applied to the connected ring $R_m$).
Then $\bar H_{R_m}\le S$. The embeddings $X_\alpha$, $\alpha\in\Phi(S,G_{R_m})$, are $S$-equivariant,
hence they are $\bar H_{R_m}$-equivariant. Since $H\le L$ preserves the subschemes $U_{P^\pm}$,
and $\Cent(G)\le L$, this implies that the embeddings $X_\alpha$ are $H_{R_m}$-equivariant.

By definition of $P=P^+$ and $P^-$, there is an isomorphism $X^*(H)\cong \ZZ$ such that
$$
\Lie(U_{P})=\bigoplus_{n>0}\Lie(G)_n\quad\mbox{and}\quad\Lie(U_{P^-})=\bigoplus_{n<0}\Lie(G)_n.
$$
Since the embeddings $X_\alpha$, $\alpha\in\Phi_P$, are $H_{R_m}$-equivariant, for any $R_m$-algebra $R'$, any $s\in H(R')$,
and any $u\in R'\otimes_{R_m} V_\alpha=R'\otimes_{R_m}\Lie(G)_\alpha$ we have
$$
sX_\alpha(u)s^{-1}=X_\alpha(s(u)).
$$
By Lemma~\ref{lem:EPgen} for any $\alpha\in\Phi_P^+$ we have $u\in\Lie(U_P)(R')$, hence $s(u)=s^nu$ for some $n=n(u)>0$.
Similarly, $\alpha\in\Phi_P^-$ we have $u\in\Lie(U_{P^-})(R')$, hence $s(u)=s^{-n}u$ for some $n=n(u)>0$.

Applying this result to the ring of Laurent polynomials $R'=R_m[Z^{\pm}]$ and $s=Z\in H(R')$, we conclude that
\begin{equation}\label{eq:sUL}
\begin{array}{l}
sU_{P}(R_m)s^{-1}\subseteq U_P(R_m[Z],ZR_m[Z]);\\
sU_{P^-}(R_m)s^{-1}\subseteq U_{P^-}(R_m[Z^{-1}],Z^{-1}R_m[Z^{-1}]);\\
s|_{L(R_m)}=\id.
\end{array}
\end{equation}
In particular, one has
\begin{equation}\label{eq:sg}
sgs^{-1}\in U_P(R_m[Z],ZR_m[Z])\le G(R_m[Z],ZR_m[Z])\le\GL_n(R_m[Z],ZR_m[Z]).
\end{equation}

On the other hand, the analogs of~\eqref{eq:sUL} hold for $\GL_n$, $U_{Q^\pm}$, and $M$ in place of
$G$, $U_{P^\pm}$, and $L$. Therefore by~\eqref{eq:gOmegaQ} we have
$$
sgs^{-1}\in U_{Q^+}(R_m[Z],ZR_m[Z])\cdot M(R_m)\cdot U_{Q^-}(R_m[Z^{-1}],Z^{-1}R_m[Z^{-1}]).
$$
Since one has
$$
\begin{array}{l}
U_{Q^+}(R_m[Z],ZR_m[Z])\cdot M(R_m)\cdot U_{Q^-}(R_m[Z^{-1}],Z^{-1}R_m[Z^{-1}])\cap
\GL_n(R_m[Z],ZR_m[Z])\\
=U_{Q^+}(R_m[Z],ZR_m[Z]),
\end{array}
$$
we conclude that $sgs^{-1}\in U_{Q^+}(R_m[Z]),ZR_m[Z])$ and thus $g\in U_{Q^+}(R_m)$, as required.

\end{proof}


\begin{lem}\label{lem:Nis_loc_glob}
Let $G$ be an isotropic reductive group scheme over a connected Noetherian commutative ring $B$, provided with a closed $B$-embedding
$G\to\GL_{n,B}$, $n\ge 3$, which is a $B$-group scheme homomorphism.
Assume that $G$ contains a non-central 1-dimensional split subtorus $\Gm_{,B}$, and let $P=P^+$ and $P^-$
be the corresponding pair of opposite parabolic subgroups that exist by Lemma~\ref{lem:T-P}.
Assume moreover that $B$ is a subring of a commutative ring $A$, and let
$h\in B$ be a non-nilpotent element. Denote by $F_h:G(A)\to G(A_h)$ the localization homomorphism.

If $Ah+B=A$, i.e. the natural map $B\to A/Ah$ is surjective, then for any $x\in E_P(A_h)$ there exist $y\in G(A)$ and $z\in E_P(B_h)$ such that
$x=F_h(y)z$.

\end{lem}
\begin{proof}

Since the ring $B$ is connected, by Lemmas~\ref{lem:relroots} and~\ref{lem:relschemes}
 the group $G$ over $B$ with a parabolic subgroup $P$ is provided with a split $B$-torus $S\le G^{\ad}$, the corresponding system
of relative roots $\Phi(S,G)=\Phi_P$ and relative root subschemes $X_\alpha(V_\alpha)$, where $\alpha\in\Phi_P$
and each $V_\alpha$ is a finitely generated projective $B$-module. By Lemma~\ref{lem:EPgen} one has
$$
E_P(R)=\left<X_\alpha(R\otimes_B V_\alpha),\ \alpha\in\Phi_P\right>
$$
for any $B$-algebra $R$.

One has $x=\prod\limits_{i=1}^mX_{\beta_i}(c_i)$, $c_i\in A_h\otimes_B V_{\beta_i}$, $\beta_i\in\Phi_P$.
 We need to show that $x\in F_h(G(A))E_P(B_h)$. Clearly, it is enough to show that
\begin{equation}\label{eq:EBh}
E_P(B_h)X_\beta(c)\subseteq F_h(G(A))E_P(B_h)
\end{equation}
for any $\beta\in\Phi_P$ and $c\in A_h\otimes_B V_\beta$.
We can assume that $\beta$ is a positive relative root without loss of generality.
We prove the inclusion~\eqref{eq:EBh} by descending induction on the height of $\beta$.
Let $e_1,\ldots,e_k$ be a set of generators of the $B$-module $V_\beta$.

Take any $z\in E_P(B_h)$. By Lemma~\ref{lem:EEn} we have $E_P(R)\le E_n(R)$ for any $B$-algebra $R$.
Take $R=A[Z]$, the ring of polynomials over $A$. For any $N\ge 1$ and $1\le i\le k$ one has
$$
zX_\beta(h^NZe_i)z^{-1}\in zE_n(A_h[Z],ZA_h[Z])z^{-1}\cap G(A_h[Z]).
$$
Since $z\in E_P(B_h)\le E_n(A_h)$, by~\cite[Lemma 3.3]{Sus} there exists $N_i\ge 1$ and $g_i(Z)\in E_n(A[Z],ZA[Z])$
such that $F_h(g_i(Z))=zX_\beta(h^{N_i}Ze_i)z^{-1}$. By~\cite[Lemma 3.5.4]{Mo} there is $K_i\ge 1$
such that $g_i(h^{K_i}Z)\in G(A[Z])$. Summing up, we conclude that there is $N\ge 1$ such that
\begin{equation}\label{eq:choiceN}
zX_\beta(h^NZe_i)z^{-1}\in F_h\bigl(G(A[Z])\bigr)
\end{equation}
for any $1\le i\le k$.

On the other hand, note that $Ah+B=A$ implies $Ah^n+B=A$ for any $n\ge 1$. Let $M\ge 0$ be such that
$h^Mc\in A\otimes_B V_\beta$.
Then one can find
$a\in A\otimes_B V_{\beta}$ and $b\in V_{\beta}$ such that
$$
c=ah^N+h^{-M}b.
$$
Write $a=\sum_{i=1}^ka_ie_i$, where $a_i\in A$.
By the multiplication formula for relative root elements~\eqref{eq:sum} we have
$$
X_{\beta}(c)=X_{\beta}(ah^{N})X_{\beta}(h^{-M}b)\prod\limits_{j\ge 2}X_{j\beta}\left(u_j\right),
$$
where $u_j=q^j_\beta(h^Na,h^{-M}b)\in A_h\otimes_B V_{j\beta}$, and, similarly,
$$
X_{\beta}(ah^{N})=\prod_{i=1}^k X_{\beta}(a_ih^{N}e_i)\prod\limits_{j\ge 2}X_{j\beta}(v_j),
$$
where $v_j\in A\otimes_B V_{j\beta}$.
By the choice of $N$ in~\eqref{eq:choiceN}, one has
\begin{equation}\label{eq:zaiei}
z\left(\prod_{i=1}^k X_{\beta}(a_ih^{N}e_i)\right) z^{-1}\in F_h(G(A)).
\end{equation}
It remains to note that, since the height of the relative roots $j\beta$, $j\ge 2$, is larger than that of $\beta$, the inductive
hypothesis version of the inclusion~\eqref{eq:EBh} can be applied to all elements $X_{j\beta}\left(v_j\right)$
and $X_{j\beta}\left(u_j\right)$, $j\ge 2$. Since, moreover, $z$ and $X_{\beta}(h^{-M}b)$ belong to $E_P(B_h)$, we see that
$$
z\left(\prod\limits_{j\ge 2}X_{j\beta}(v_j)\right)X_{\beta}(h^{-M}b)\left(\prod\limits_{j\ge 2}
X_{j\beta}\left(u_j\right)\right)z^{-1}\in F_h(G(A))E_P(B_h).
$$
Combining this result with~\eqref{eq:zaiei}, we conclude that
$$
zX_{\beta}(c)z^{-1}\in F_h\left(G(A)\right)E_P(B_h),
$$
which proves~\eqref{eq:EBh}.
\end{proof}

\begin{lem}\label{lem:semisimple-hdvr}
Let $R$ be a henselian discrete valuation ring. Let $K$ be the field of fractions
of $R$. Let $G$ be a semisimple simply connected $R$-group scheme
such that every semisimple normal $R$-subgroup scheme of $G$
contains a split $R$-torus $\mathbb G_{m,R}$.
Then $G$ contains a strictly proper parabolic $R$-subgroup $P$, and
$$
G(K)=G(R)E_P(K).
$$
\end{lem}
\begin{proof}
Since $G$ is a semisimple simply connected $R$-group scheme, by~\cite[Exp. XXIV 5.3, Prop. 5.10]{SGA3}
there exist
finite \'{e}tale ring extensions
$R^{\prime}_i/R$, $1\le i\le n$, and absolutely almost simple simply connected $R'_i$-group schemes $G'_i$
such that
$$
G\cong G_1\times_{\Spec R} G_2\times_{\Spec R}\ldots\times_{\Spec R} G_n,
$$
where $G_i=\R_{R^{\prime}_i/R}(G^{\prime}_i)$ are minimal semisimple normal subgroups of $G$.
Clearly,
\begin{equation}\label{eq:Gprod}
G(K)\cong\prod\limits_{i=1}^nG'_i(K\otimes_R R_i)\quad\mbox{and}\quad G(R)\cong\prod\limits_{i=1}^nG'_i(R_i).
\end{equation}

Since each $G_i$ contains $\Gm_{,R}$, one readily sees
that each $G'_i$ is isotropic, i.e. contains $\Gm_{,R'_i}$ (see e.g. the proof of~\cite[Theorem 11.1]{PSV}).
Hence by Lemma~\ref{lem:T-P} $G'_i$ has a proper parabolic $R'_i$-subgroup $P'_i$.
Then $P_i=\R_{R'_i/R}(P'_i)$ is a proper parabolic $R$-subgroup of $G_i$, and
$$
P=P_1\times_{\Spec R} P_2\times_{\Spec R}\ldots\times_{\Spec R} P_n
$$
is a strictly proper parabolic $R$-subgroup of $G$. We have
\begin{equation}\label{eq:Eprod}
E_P(K)=\prod_{i=1}^n E_{P_i}(K)\cong\prod_{i=1}^n E_{P'_i}(K\otimes_R R'_i).
\end{equation}

Fix an $i$, $1\le i\le n$, and abbreviate $A=R'_i$, $H=G'_i$, $P'=P'_j$. Since the map $R\to A$ is finite \'etale, the ring $A$
is a product of a finite number of henselian discrete valuation rings $A_j$, $1\le j\le n$, and
$K\otimes_R A$ is the product of their respective fraction fields $L_j$, $1\le j\le n$. By~\cite[Lemme 4.5]{Gi}
one has
$$
H(K\otimes_R A)=\prod\limits_{j=1}^m H(L_j)=\prod\limits_{j=1}^m H(A_j)E_{P'}(L_j)=H(A)E_{P'}(K\otimes_R A).
$$
Combining this result with~\eqref{eq:Gprod} and~\eqref{eq:Eprod}, we deduce that
$$
G(K)\cong \prod_{i=1}^n G'_i(R'_i)E_{P'_i}(K\otimes_R R'_i)\cong\prod_{i=1}^n G_i(R)E_{P_i}(K)=G(R)E_P(K),
$$
as required.

\end{proof}

\begin{thm}
\label{th:main}
Let $R$ be a semi-local Dedekind domain. Let $K$ be the field of fractions
of $R$.
Let $G$ be a
reductive semisimple simply connected $R$-group scheme
such that every semisimple normal $R$-subgroup scheme of $G$
contains a split $R$-torus $\mathbb G_{m,R}$.
Then the map
$$
H^1_{\text{\'et}}(R,G)\to H^1_{\text{\'et}}(K,G)
$$
of pointed sets induced by the inclusion of $R$ into $K$ has trivial kernel.
\end{thm}

\begin{proof}
We prove the theorem by induction on the number of maximal ideals in $R$.
If $R$ is local then the theorem holds by~\cite{Ni}. Let $n>1$ be an integer and suppose the theorem holds for all
Dedekind domains containing strictly less than $n$ maximal ideals. Prove
that the theorem holds for a Dedekind domain $R$ with exactly $n$ maximal ideals.
Let $\mathfrak m \subset R$ be a maximal ideal and let $f\in \mathfrak m$
be its generator. Let $R'$ be the Henselization of $R$ at the maximal ideal $\mathfrak m$
and let $R_f$ be the localization of $R$ at $f$. Let $L'$ be the fraction field of $R'$.

Let $\mathcal E$ be a principal $G$-bundle over $R$ which is trivial over the field $K$.
By the inductive hypothesis $\mathcal E$ is trivial over $R_f$
and over $R'$. Thus we may assume that $\mathcal E$ is obtained by
patching over $\Spec L'$ of two trivial principal $G$-bundles
$G_f:=G\times_{\Spec R}\Spec R_f$ and $G':=G\times_{\Spec R} \Spec R'$
using an element $x \in G(L')$.

By Lemma~\ref{lem:semisimple-hdvr} $G$ contains a strictly proper parabolic $R$-subgroup $P$,
and one has
$$
G(L')=G(R').E_P(L')
$$
So, $x=x''.x'$ for some $x''\in G(R')$ and $x' \in E_P(L')$.
Replacing the patching element $x=x''.x'$ with $x' \in E_P(L')$
we do not change the isomorphism class of the principal $G$-bundle $\mathcal E$ over $R$.
Moreover, by Lemma \ref{lem:Nis_loc_glob} one can present $x'$ in the form
$x'=y.z$ with $y\in G(R')$ and $z\in E_P(R_f)$. The latter yields
the triviality of the principal $G$-bundle $\mathcal E$ over $R$, since $\Spec R_f$ and $\Spec R'$ form
a covering of $\Spec R$ for the Nisnevich topology.

\end{proof}

\renewcommand{\refname}{References}

\end{document}